\documentclass[11pt,a4paper]{amsart}
\usepackage{lineno,hyperref}
\usepackage{amssymb}
\usepackage{latexsym, amsmath, amscd,amsthm}
\usepackage{graphicx, kotex, multirow, array, caption, xcolor ,MnSymbol, arydshln}

\newtheorem{theorem}{Theorem}[section]
\newtheorem{corollary}[theorem]{Corollary}
\newtheorem{proposition}[theorem]{Proposition}
\newtheorem{lemma}[theorem]{Lemma}

\theoremstyle{definition}
\newtheorem{definition}[theorem]{Definition}

\begin{document}
	
\title{On the Laplacian spectrum of $k$-symmetric graphs}
	
\author[S. Moon]{Sunyo Moon}
\address{Research Institute for Natural Sciences, Hanyang University, Seoul 04763, Korea}
\email{symoon89@hanyang.ac.kr}

\author[H. Yoo]{Hyungkee Yoo}
\address{Institute of Mathematical Sciences, Ewha Womans University, Seoul 03760, Korea}
\email{hyungkee@ewha.ac.kr}



\thanks{The second author(Hyungkee Yoo) was supported by the National Research Foundation of Korea(NRF) grant funded by the Korea government Ministry of Education(NRF-2019R1A6A1A11051177) and Ministry of Science and ICT(NRF-2022R1A2C1003203).}

\begin{abstract}
For some positive integer $k$,
if the finite cyclic group $\mathbb{Z}_k$ can act freely on a graph $G$,
then we say that $G$ is $k$-symmetric.
In 1985, Faria showed that the multiplicity of Laplacian eigenvalue 1 is greater than or equal to the difference between the number of pendant vertices and the number of quasi-pendant vertices.
But if a graph has a pendant vertex, then it is at most 1-connected.
In this paper, we investigate a class of 2-connected $k$-symmetric graphs with a Laplacian eigenvalue 1.
We also identify a class of $k$-symmetric graphs in which all Laplacian eigenvalues are integers.
\end{abstract}

\maketitle

\section{Introduction} \label{sec:intro}
A {\it simple graph} $G=(V,E)$ is a combinatorial object consisting of a finite set $V$ and a set $E$ of unordered pairs of different elements of $V$.
The elements of $V$ and $E$ are called the {\it vertices} and the {\it edges} of the graph $G$, respectively.
For a given graph $G$,
the vertex set and the edge set of $G$ are denoted by $V(G)$ and $E(G)$, respectively.

Let $G$ be a graph with enumerated vertices.
The {\it Laplacian matrix} $L(G)$ of $G$ is defined as $L(G)=D(G)-A(G)$, where $D(G)$ is the diagonal matrix of vertex degrees and $A(G)$ is the adjacency matrix of $G$.
Thus the Laplacian matrix is symmetric.
Note that the Laplacian matrix can be considered a positive-semidefinite quadratic form on the Hilbert space generated by $V(G)$.
Since the Laplacian matrix contains information on the structure of the graph, it has been studied importantly in various applied fields including artificial neural network research using graph shaped data~\cite{Kipf1,Kipf2}.

Let $G$ be a graph with $n$ vertices.
For a square matrix $M$, we denote the characteristic polynomial of $M$ by $\mu(M,x)$. 
A root of the characteristic polynomial of Laplacian matrix $L(G)$ is called a {\it Laplacian eigenvalue} of $G$.
Denote the all eigenvalues of $L(G)$ by
$\lambda_{n}(G) \le \lambda_{n-1}(G) \le \cdots \le \lambda_1(G)$.
It is well-known that $\lambda_n(G)=0$ and $\lambda_1(G) \le n$.
The multiset of Laplacian eigenvalues of $G$ is called the {\it Laplacian spectrum} of $G$.
The Laplacian spectrum of the complement graph $\overline{G}$ of $G$ is satisfying
$$
0=\lambda_{n}(\overline{G}) \le n-\lambda_1(G) \le \cdots \le n-\lambda_{n-1}(G).
$$
The Laplacian spectrum shows us several properties of the graph.
For instance, Kirchhoff \cite{Kirchhoff} proved that the number of spanning tree of a connected graph $G$ with $n$ vertices is $\frac{1}{n}\lambda_1(G)\cdots\lambda_{n-1}(G)$.
Let $m_G(\lambda)$ denote the multiplicity of $\lambda$ as a Laplacian eigenvalue of $G$.
Note that the multiplcity of 0 is equal to the number of connected components of $G$.

The {\it connectivity} $\kappa(G)$ of a graph $G$ is the minimum number of vertices whose removal results in a disconnected or trivial graph. A graph $G$ is said to be {\it $t$-connected} if $\kappa(G) \geq t$.
If a graph is $t$-connected, then it is $(t\!-\!1)$-connected.
Fiedler~\cite{Fiedler} proved that the second smallest Laplacian eigenvalue of $G$ is less than or equal to $\kappa(G)$.

A {\it pendant vertex} of $G$ is a vertex of degree $1$.
A {\it quasi-pendant} of $G$ is a vertex adjacent to a pendant.
We denote the number of pendants of $G$ by $p(G)$,
and the number of quasi-pendant vertices by $q(G)$.
In \cite{Faria}, Faria showed that for any graph $G$, 
$$
m_G(1)\geq p(G)-q(G).
$$
It implies that if $p(G)$ is greater than $q(G)$, then $G$ has a Laplacian eigenvalue~1.
Also, such graph $G$ is at most 1-connected.
In \cite{Barik}, Barik et al. found trees with a Laplacian eigenvalue 1 even though the right-hand side of the above inequality is 0.
Since a tree has connectivity 1, we focus on 2-connected graph with a Laplacian eigenvalue 1.

The simplest way to obtain a 2-connected graph with a Laplacian eigenvalue 1 is the Cartesian product.
The {\it Cartesian product} $G \square H$ of graphs $G$ and $H$ is the graph with the vertex set $V(G) \times V(H)$
such that two vertices $(v,v')$ and $(w,w')$ are adjacent if $v=w$ and $v'$ is adjacent to $w'$ in $H$,
or if $v'=w'$ and $v$ is adjacent to $w$ in $G$.
Fiedler~\cite{Fiedler} showed that the Laplacian eigenvalues of the Cartesian product $G \square H$ are all possible sums of Laplacian eigenvalues of $G$ and $H$.
If either $G$ or $H$ has a Laplacian eigenvalue 1,
then 1 is a Laplacian eigenvalue of $G \square H$.
\v{S}pacapan~\cite{Spacapan} showed that
the connectivity of $G \square H$ is
$$
\kappa(G \square H) = \min\{ \kappa(G)|H|, \kappa(H)|G|, \delta(G \square H) \},
$$
where $\delta(G \square H)$ is the minimum degree of $G \square H$.
Remark that if $G$ and $H$ are connected graphs,
then $G \square H$ is 2-connected.
Thus we concentrate a 2-connected graph that does not decompose nontrivial graphs under the Cartesian product.
If a graph does not admit the nontrivial Cartesian product decomposition,
then the graph is called {\it prime} with respect to the Cartesian product. 
In this paper, we prove the following theorem.

\begin{theorem} \label{thm:main}
For any positive integer $n$, there is a 2-connected prime graph $G$ with respect to the Cartesian product,
$$
m_G(1) \ge n.
$$
\end{theorem}

Meanwhile, integral spectra of Laplacian matrix or adjacency matrix are studied in various application fields including physics and chemistry~\cite{Christandl, Cvetkovic1, Cvetkovic2}.
A graph with integral Laplacian spectrum is called a {\it Laplacian integral graph}.
If a graph does not include the path $P_4$ as an induced subgraph,
then it is called a {\it cograph}.
In \cite{Merris}, Merris showed that a cograph is a Laplacian integral graph.
Many researchers~\cite{Del-Vecchio,Grone,Kirkland1,Kirkland2} have investigated infinitely many classes of Laplacian integral graphs that are not cographs.
In section \ref{sec:laplacian}, we introduce a new graph $\mathcal{C}(n,m)$ for some positive integers $n$ and $m$.
The graph $\mathcal{C}(n,m)$ is obtained by connecting several set of vertices for $m$ parallel copies of $n$-complete graph $K_n$ to the corresponding vertex of $\overline{K}_n$.
Later, we examine that if $n \geq 2$, then $\mathcal{C}(n,m)$ has the path $P_4$ as the induced subgraph,
that is, it is not a cograph.
In this paper, we also prove the following theorem.

\begin{theorem} \label{thm:main2}
There are infinitely many pairs of positive integers $n$ and $m$,
which make $\mathcal{C}(n,m)$ a Laplacian integral graph.
\end{theorem}

This paper is organized as follows.
In Section~\ref{sec:pre}, We provide some linear algebra results needed for proof of main theorems.
In Section~\ref{sec:symm}, We define the $k$-symmetric graph by relaxing the condition of symmetric graph, and examine its properties.
In Section~\ref{sec:mult} and Section~\ref{sec:laplacian},
we prove the main theorems and related properties.


\section{Preliminaries} \label{sec:pre}
In this section, we introduce some definitions and properties that will be used in this paper.
The set of all $m \times n$ matrices over a field $\mathbb{F}$ is denoted by $M_{m\times n}(\mathbb{F})$. Denote $M_{n\times n}(\mathbb{F})$ by $M_n(\mathbb{F})$.
We denote by $I_n$ and $J_n$ the $n \times n$ identity matrix and the $n \times n$ matrix whose entries are ones. 
Also, $1_n$ is the $n$-vector of all ones.

Let $A \in M_n(\mathbb{F})$ be a block matrix of the form
  \begin{equation*}
      A= \begin{pmatrix}
       A_{11} & A_{12} \\ A_{21} & A_{22}
      \end{pmatrix},
  \end{equation*}
where $A_{11} \in M_{m}(\mathbb{F})$, $A_{12} \in M_{m \times (n-m)}(\mathbb{F})$, $A_{21} \in M_{(n-m) \times m}(\mathbb{F})$ and $A_{22}\in M_{n-m}(\mathbb{F})$.
It is well known that if $A_{22}$ is invertible, then $\det{A}=\det{A_{22}}\det(A_{11}-A_{12}A_{22}^{-1}A_{21})$ (see \cite[Chapter 0]{Horn2}).

For two matrices $A=(a_{ij}) \in M_{m\times n}(\mathbb{F})$ and $B \in M_{p\times q}(\mathbb{F})$, the {\it Kronecker product} of $A$ and $B$, denoted by $A\otimes B$, is defined as 
\begin{equation*}
    A\otimes B=\begin{pmatrix}
        a_{11}B & a_{12}B & \cdots & a_{1n}B\\
        a_{21}B & a_{22}B & \cdots & a_{2n}B\\
        \vdots & \vdots & \ddots & \vdots\\
        a_{m1}B & a_{m2}B & \cdots & a_{mn}B\\
    \end{pmatrix}.
\end{equation*}
We state some basic properties of the Kronecker product (for more details, see \cite[Chapter 4]{Horn1}):
\begin{itemize}
    \item[(a)] $A\otimes(B+C)=A\otimes B + A\otimes C$.
    \item[(b)] $(B+C) \otimes A = B\otimes A+ C\otimes A$.
    \item[(c)] $(A\otimes B)(C \otimes D)=AC \otimes BD$.
    \item[(d)] If $A\in M_m(\mathbb{F})$ and $B\in M_n(\mathbb{F})$ are invertible, then $(A \otimes B)^{-1}=A^{-1}\otimes B^{-1}$.
    \item[(e)] $\det(A\otimes B)=(\det{A})^n(\det{B})^m$ for $A \in M_m(\mathbb{F})$ and $B \in M_n(\mathbb{F})$.
\end{itemize}

A matrix $T\in M_n(\mathbb{F})$ of the form
\begin{equation*}
    T=\begin{pmatrix}
        a_0 & a_{-1} & a_{-2} & \cdots & a_{-(n-1)}\\
        a_1 & a_0 & a_{-1} & \cdots & a_{-(n-2)}\\
        a_2 & a_1 & a_0 & \cdots & a_{-(n-3)}\\
        \vdots & \vdots & \vdots & \ddots &\vdots \\
        a_{n-1} & a_{n-2} & a_{n-3} & \cdots & a_0\\
    \end{pmatrix}
\end{equation*}
is called a {\it Toeplitz matrix}.
In \cite{Lv}, the authors gave a Toeplitz matrix inversion formula.
\begin{theorem}[\cite{Lv}, Theorem 1]\label{thm:Toe}
  Let $T=(a_{i-j})^n_{i,j=1}$ be a Toeplitz matrix and let $f=(0,a_{n-1}-a_{-1},\cdots,a_{2}-a_{-(n-2)},a_{1}-a_{-(n-1)})^T$ and $e_1=(1,0, \cdots, 0)^T$. If each of the systems of equations $Tx=f$, $Ty=e_1$ is solvable, $x=(x_1,x_2,\ldots,x_n)^T$, $y=(y_1,y_2,\ldots,y_n)^T$, then
  \begin{itemize}
      \item[(a)] $T$ is invertible;
      \item[(b)] $T^{-1}=T_1U_1+T_2U_2$, where
      \begin{equation*}
          T_1=\begin{pmatrix}
              y_1 & y_n & \cdots & y_2\\
              y_2 &  y_1 & \ddots & \vdots \\
              \vdots & \ddots & \ddots &y_n\\
              y_n & \cdots & y_2 & y_1
          \end{pmatrix}
          ,
          \qquad
          U_1=\begin{pmatrix}
              1 & -x_n & \cdots & -x_2\\
              0 & 1 & \ddots & \vdots \\
              \vdots & \ddots & \ddots & -x_n\\
              0 & \cdots &  0& 1
          \end{pmatrix},
      \end{equation*}
      \begin{equation*}
          T_2=\begin{pmatrix}
              x_1 & x_n & \cdots & x_2\\
              x_2 &  x_1 & \ddots & \vdots \\
              \vdots & \ddots & \ddots &x_n\\
              x_n & \cdots & x_2 & x_1
          \end{pmatrix}
          ,~~\text{and}~~
          U_2=\begin{pmatrix}
              0 & y_n & \cdots & y_2\\
              0 & 0 & \ddots & \vdots \\
              \vdots & \ddots & \ddots & y_n\\
              0 & \cdots &  0& 0
          \end{pmatrix}
            .
      \end{equation*}
  \end{itemize}
\end{theorem}
\begin{corollary}\label{coro:detinv}
  Let $aI_n+bJ_n$ be a matrix in $M_n(\mathbb{F})$.
  Then 
  \begin{itemize}
      \item[(a)]$\det(aI_n+bJ_n)=a^{n-1}(a+nb)$.
      \item[(b)] If $aI_n+bJ_n$ is invertible, then its inverse matrix is  $$\frac{1}{a(a+nb)}((a+nb)I_n-bJ_n).$$
  \end{itemize}
\end{corollary}
\begin{proof}
  \begin{itemize}
    \item[(a)] It is easy to check that
        \begin{equation*}
            \det \begin{pmatrix}
                  a+b & b & b & \cdots & b \\
                  b & a+b & b & \cdots & b \\
                  b & b & a+b & \cdots & b \\
                  \vdots & \vdots & \vdots & \ddots& \vdots\\
                  b & b & b & \cdots & a+b \\ 
              \end{pmatrix}
              =
              \det
              \begin{pmatrix}
                  a & 0 & 0 & \cdots & 0 \\
                  0 & a & 0 & \cdots & 0 \\
                  0 & 0 & a & \cdots & 0 \\
                  \vdots & \vdots & \vdots & \ddots& \vdots\\
                  b & 2b & 3b & \cdots & a+nb \\ 
              \end{pmatrix}.
          \end{equation*}
        Hence the determinant of $aI_n+bJ_n$ is $a^{n-1}(a+nb)$. 
    \item[(b)] Note that the matrix $aI_n+bJ_n$ is Toeplitz. Let 
    $$x=(0,\ldots,0)^T \text{~~and~~} y=\bigg(\frac{a+nb-b}{a(a+nb)},\frac{-b}{a(a+nb)},\ldots,\frac{-b}{a(a+nb)}\bigg)^T.$$
        Then $(aI_n+bJ_n)x=0$ and $(aI_n+bJ_n)y=e_1$.
        By Theorem \ref{thm:Toe}, 
         the inverse of $aI_n+bJ_n$ is $$\frac{1}{a(a+nb)}((a+nb)I_n-bJ_n).$$
  \end{itemize}
\end{proof}


\section{k-Symmetric graphs} \label{sec:symm}

Symmetry is an important property of graphs.
We deal with graphs that has symmetric property.
Let $G$ be a graph.
An {\it automorphism} $\varphi$ of a graph $G$ is a permutation of $V(G)$ such that
$\varphi(v)$ and $\varphi(w)$ are adjacent if and only if $v$ and $w$ are adjacent
where $v$ and $w$ are vertices of $G$.
The set of all automorphisms of $G$ is called an {\it automorphism group} of $G$ and denoted by $\text{Aut}(G)$.
A graph $G$ is {\it symmetric}
if $\text{Aut}(G)$ acts transitively on both vertices of $G$ and ordered pairs of adjacent vertices.
This implies that $G$ is {\it regular},
that is, all vertices have the same degree.
However, it is a very difficult problem to determine whether a given graph is a symmetric graph.
Thus we concentrate the cyclic part of $\text{Aut}(G)$.
In this section, we define $k$-symmetric graphs and give some their properties.
Also, we construct a $k$-symmetric graph from other $k$-symmetric graphs.

\begin{definition}
  Let $k$ be a positive integer.
  A graph $G$ is {\it $k$-symmetric}
  if there is a subgroup $\mathcal{H}$ of $\text{Aut}(G)$ such that
  $\mathcal{H}$ is isomorphic to $\mathbb{Z}_k$
  and $\mathcal{H}$ freely act on vertices.
  A generator of $\mathcal{H}$ is called a {\it $k$-symmetric automorphism}.
\end{definition}

The above definition tells us that all graphs are 1-symmetric because the trivial group freely acts on any graph.
If a graph $G$ with $n$ vertices is $n$-symmetric,
then the automorphism group $\text{Aut}(G)$ has a cyclic subgroup $\mathcal{H}$ which transitively acts on vertices.
Thus $G$ is regular.
However, the converse is not true even though $G$ is a symmetric graph.
Before examining this, we check the following proposition.

\begin{proposition}
Let $G$ be a graph with $n$ vertices.
If $G$ is $n$-symmetric,
then either $G$ or its complement $\overline{G}$ have a Hamiltonian cycle.
\end{proposition}

\begin{proof}
Let $G$ be an $n$-symmetric graph with $n$ vertices,
and let $\varphi$ be an $n$-symmetric automorphism of $G$.
Choose a vertex $v$.
If $v$ and $\varphi(v)$ are adjacent,
then $\varphi^i(v)$ and $\varphi^{i+1}(v)$ are also adjacent for any interger $i$.
Since $G$ is $n$-symmetric,
the group generated by $\varphi$ acts freely and transitively on $V(G)$.
Thus the sequence $v, \varphi(v), \dots, \varphi^n(v)$ induces a Hamiltonian cycle of $G$.
Suppose that $v$ and $\varphi(v)$ are not adjacent in $G$.
Then $v$ and $\varphi(v)$ are adjacent in $\overline{G}$.
Hence the sequence of vertices induces a Hamiltonian cycle of $\overline{G}$.
\end{proof}

For example, the Petersen graph in Figure~\ref{fig:petersen} is 5-symmetric because the 5-fold rotation satisfies the 5-symmetric automorphism condition.
The Petersen graph is a symmetric graph with 10 vertices.
But since the Petersen graph is not Hamiltonian,
it is not 10-symmetric.
For any positive integer $k$,
$k$-symmetric graphs are satisfying the following properties.

\begin{figure}[h!]
\centering
\includegraphics[width=0.55\textwidth]{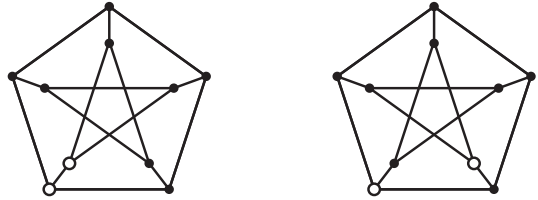}
\caption{The Petersen graphs with different bases for the 5-fold rotation}
\label{fig:petersen}
\end{figure}

\begin{proposition}
    Let $G$ be a $k$-symmetric graph for some integer $k$ and let $d$ be a divisor of $k$.
    Then $G$ is a $d$-symmetric graph.
\end{proposition}

\begin{proof}
  Let $\varphi$ be a $k$-symmetric automorphism of $G$ and let $k=k'd$ for some integer $k'$.
  Define an automorphism $\psi$ of $G$ by $\psi=\varphi^{k'}$.
  Then $\psi^d(v)=\varphi^{k'd}(v)=\varphi^k(v)=\text{id}_G(v)$ for all $v \in V(G)$.
  The subgroup $\langle \psi\rangle$ of $\text{Aut}(G)$ is isomorphic to $\mathbb{Z}_d$.
  Thus $G$ is a $d$-symmetric graph.
\end{proof}

\begin{proposition}\label{prop:union}
  Let $G_1$ and $G_2$ be $k$-symmetric graphs for some integer $k$. Then $G_1 \cup G_2$ is $k$-symmetric graph.
\end{proposition}
\begin{proof}
  Let $\varphi_1$ and $\varphi_2$ be $k$-symmetric automorphisms of $G_1$ and $G_2$, respectively. Then the automorphism $\varphi_1+\varphi_2$ of $G_1 \cup G_2$ is defined by 
  \begin{equation*}
     (\varphi_1+\varphi_2)(v)=\left\{\begin{array}{ll}
          \varphi_1(v), &\mbox{if $v \in V(G_1)$},  \\
          \varphi_2(v), &\mbox{if $v \in V(G_2)$} 
     \end{array}\right.  .
  \end{equation*}
  Hence $G_1 \cup G_2$ is a $k$-symmetric graph.

\end{proof}

 Let $\varphi$ be a $k$-symmetric automorphism of a graph $G$ and let $\operatorname{id}_H$ be the identity automorphism of a graph $H$. Then the automorphism $\varphi \times \operatorname{id}_H$ of $G \square H$ is $k$-symmetric. Thus we obtain the following proposition.

\begin{proposition}\label{prop:cartesian}
  Let $G$ be a $k$-symmetric graphs for some integer $k$. For any graph $H$, the Cartesian product $G \square H$ is $k$-symmetric graph.
\end{proposition}

Let $G$ be a graph with a $k$-symmetric automorphism $\varphi$.
Then $\mathbb{Z}_k$ acts on $V(G)$ as follows.
For any $i \in \mathbb{Z}_k$ and $v \in V(G)$,
we define $i \cdot v = \varphi^i(v)$.
For any vertex $v$,
the orbit of $v$ is denoted by $\mathbb{Z}_k \cdot v$.
Let $B_{\varphi}$ be a minimal subset of $V(G)$ such that
$$
\bigcup_{i=0}^{k-1} \varphi^{i}(B_{\varphi})=V(G).
$$
Alternatively,
$B_{\varphi}$ is a minimal subset of $V(G)$ such that
$$
\bigcup_{v \in B_{\varphi}} \mathbb{Z}_k \cdot v=V(G).
$$
The set $B_{\varphi}$ is called a {\it base} of $\varphi$.
Since $k$ choices are possible for each orbit,
$B_{\varphi}$ is not unique as drawn in Figure~\ref{fig:petersen}.
Note that the size of the base $B_\varphi$ is $\frac{|V(G)|}{k}$.

Now we introduce how to construct a $k$-symmetric graph from other $k$-symmetric graphs for any positive integer $k$.
First we observe a graph join.
Let $H_1$ and $H_2$ be graphs.
The {\it graph join} $H_1 \! \vee \! H_2$ of $H_1$ and $H_2$
is a graph obtained by joining each vertex of $H_1$ to all vertices of $H_2$.
Since every graph is 1-symmetric with respect to identity map,
we can understand graph join $H_1 \! \vee \! H_2$ as a join of the bases $V(H_1)$ and $V(H_2)$ of $\operatorname{id}_{H_1}$ and $\operatorname{id}_{H_2}$.
From this fact, we generalize graph join.
\begin{definition}\label{def:k-join}
  For $i\in \{1,2\}$,
  let $G_i$ be a $k$-symmetric graph with a $k$-symmetric automorphism $\varphi_i$,
  and let $B_i$ be a chosen base of $\varphi_i$.
  The {\it $k$-symmetric join} is a graph obtained by joining each vertex of $\varphi_1^j(B_1)$ to all vertices of $\varphi_2^j(B_2)$ for all $j \in \mathbb{Z}_k$.
  The $k$-symmetric join is denoted by $(G_1,\varphi_1,B_1) \vee_k (G_2,\varphi_2,B_2)$.
  If we choose arbitrary $k$-symmetric automorphisms and its bases of $G_1$ and $G_2$, then the $k$-symmetric join is simply denoted by $G_1 \vee_k G_2$.
\end{definition}

The $k$-symmetric join preserves the $k$-symmetry.
Because the $k$-symmetric automorphism of $G_1 \vee_k G_2$ is $\varphi_1 + \varphi_2$ and its base is $B_1 \cup B_2$.
Definition~\ref{def:k-join} derives that the graph join is the 1-symmetric join.
Note that, $n$-symmetric joins are not unique even if the base of each $G_i$ is unique.
For instance, the Cartesian product of 5-cycle $C_5$ with $K_2$ and the Petersen graph are both 5-symmetric joins of two 5-cycles, but they are not isomorphic as drawn in Figure~\ref{fig:5-join}.
\begin{figure}[h!]
\centering
\includegraphics[width=0.6\textwidth]{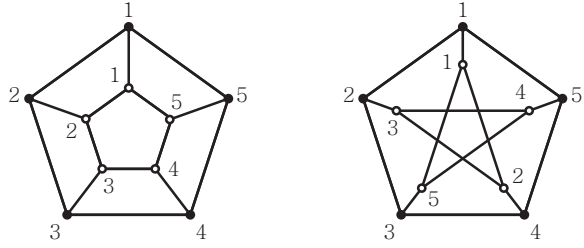}
\caption{The $5$-symmetric join of two 5-cycles is $C_5 \square K_2$ if both automorphisms are $(1,2,3,4,5)$, and the Petersen graph if automorphisms are $(1,2,3,4,5)$ and $(1,3,5,2,4)$}
\label{fig:5-join}
\end{figure}

\section{2-connected $k$-symmetric graphs with Laplacian eigenvalue 1} \label{sec:mult}

In this section,
we prove Theorem~\ref{thm:main}.
First we consider the multiplicity of an integral Laplacian eigenvalue.
Recall that for a given graph $G$, the partition $\pi=(V_1, V_2, \cdots , V_k)$ of $V(G)$ is an {\it equitable partition} if  for all $i,j \in \{1,2,\ldots,k\}$ and for any $v\in V_i$ the number $d_{ij}=|N_G(v)\cap V_j|$ depends only on $i$ and $j$. The $k\times k$ matrix $L^\pi(G)=(b_{ij})$ defined by 
\begin{equation*}
    b_{ij}=\left\{
        \begin{array}{ll}
            -d_{ij}, &\mbox{if $i\neq j$},  \\
            \sum_{s=1}^k d_{is}-d_{ij}, & \mbox{if $i=j$}
        \end{array}
      \right.
  \end{equation*}
is called the {\it divisor matrix} of $G$ with respect to $\pi$.

\begin{lemma}[\cite{Cardoso, Cvetkovic3}]\label{lem:equit}
  Let $G$ be a graph and let $\pi=(V_1,\ldots,V_k)$ be an equitable partition of $G$ with divisor matrix $L^\pi(G)$.
  Then each eigenvalue of $L^\pi(G)$ is also an eigenvalue of $L(G)$.
\end{lemma}
In the following theorem,
we obtain the multiplicity of $n$ of $k$-symmetric join of graphs where $n$ is the size of a base.

\begin{theorem}\label{thm:multm} 
  Let $G_1,\ldots,G_l$ be $k$-symmetric graphs for some $k$ and let $G = G_1 \vee_k \cdots \vee_k G_l$ be the $k$-symmetric join of $G_1,\ldots ,G_l$. Let $n=\frac{|V(G)|}{k}$. Then
  \begin{equation*}
      m_G(n)\geq l-1.
  \end{equation*}
\end{theorem}
\begin{proof}
  The partition $\pi=(V(G_1),\ldots,V(G_l))$ is an equitable partition of $G$.
  Then we have
  \begin{equation*}
     L^\pi(G)=\begin{pmatrix}
        n-n_1 & -n_2 & \cdots & -n_l\\
        -n_1 & n-n_2 & \cdots & -n_l \\
        \vdots & \vdots & \ddots & \vdots \\
        -n_1 & -n_2 & \cdots  & n-n_l  
      \end{pmatrix},
  \end{equation*}
  where $n_i=\frac{|V(G_i)|}{k}$ for $i=1,\ldots,l$. 
 Since the characteristic polynomial of $L^\pi(G)$ is $\mu(L^\pi(G),x)=x(x-n)^{l-1}$,
  by Lemma \ref{lem:equit}, we obtain
  \begin{equation*}
      m_G(n) \geq l-1.
  \end{equation*}
\end{proof}

If each $G_i$ in the above theorem is $n$-symmetric graph with $n$ vertices,
then the size of a base of $G$ is $l$.

\begin{corollary}
  Let $G_1,\ldots,G_l$ be $n$-symmetric graphs with $n$ vertices.
  Then for any their $n$-symmetric join $G$,
  \begin{equation*}
      m_G(l) \geq l-1.
  \end{equation*}
\end{corollary}

Let $G$ be an $n$-symmetric graph with $n$ vertices.
Take an $n$-symmetric automorphism $\varphi$ of $G$.
Let $G'$ be a graph that $n$-symmetric join of $m$ copies of $G$ along $\varphi$.
Then since each base of the copy of $G$ is a vertex,
the base of $G'$ induces the $m$ complete graph $K_m$.
Since $G'$  is constructed by same $n$-symmetric automorphism,
$G'$ becomes the Cartesian product of $G$ and $K_m$.

\begin{corollary}
  Let $G$ be $n$-symmetric graphs with $n$ vertices.
  Then for any positive integer $m$,
  \begin{equation*}
      m_{K_m \square G}(m) \geq m-1.
  \end{equation*}
\end{corollary}

By the \v{S}pacapan's result~\cite{Spacapan} about the connectivity of the Cartesian product in Section~\ref{sec:intro},
we realize that for any positive integer $m$, there is a $m$-connected graph $G$ with $m_G(m) \geq m-1$.

Now consider a special case of $k$-symmetric join.
For any $i \in \{1, \dots , l\}$,
let $G_i$ be a $k$-symmetric graph for some positive integer $k$ and let $\varphi_i$ be an associated $k$-symmetric automorphism.
Let $B^1_i$ be a base of $G_i$,
and let $B^j_i=\varphi^j_i(B^1_i)$.
Racall that the union of $G_1, \dots G_l$ is also $k$-symmetric with the $k$-symmetric automorphism $\varphi_1 + \cdots + \varphi_l$ and the base $B^1_1 \cup \cdots \cup B^1_l$.
Define a graph $G$ by $k$-symmetric joining 
$\overline{K}_k$ and $G_1 \cup \cdots \cup G_l$.
Then the subgraph induced by a base of $G$ has a cut-vertex as drawn in Figure \ref{fig:orb_pro} (a).
From this fact, we can take an equitable partition $\pi=(V_0,V_1,\ldots,V_l)$
where $V_0=V(\overline{K}_k)$ and $V_i=V(G_i)$ for any $i \in \{1, \dots , l\}$ as drawn in Figure \ref{fig:orb_pro} (b).
Remark that for any distinct $i$ and $i'$, there is no edge connecting two subgraphs $G_i$ and $G_{i'}$ in $G$.
To prove Theorem \ref{thm:main}, we need the following two theorems.

\begin{figure}[h!]
\centering
\includegraphics[width=0.8\textwidth]{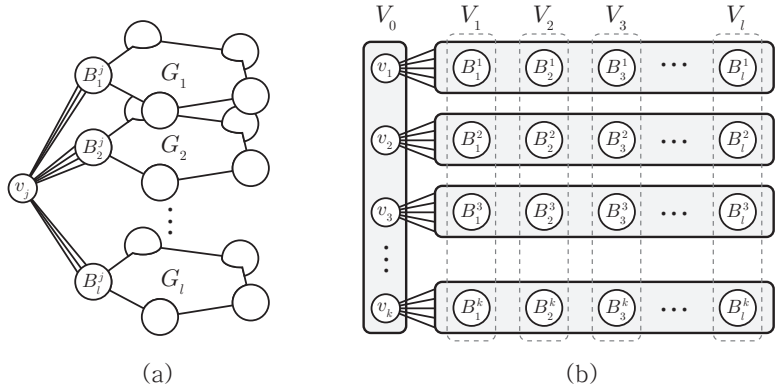}
\caption{$k$-symmetric graph $G=\overline{K}_k \vee_k \left( G_1 \cup \cdots \cup G_l \right)$}
\label{fig:orb_pro}
\end{figure}

\begin{theorem} \label{thm:orbit}
Let $G_1, \dots , G_l$ be $k$-symmetric graphs for some positive integers $l$ and $k$,
and let $G=\overline{K}_k \vee_k \left( G_1 \cup \cdots \cup G_l \right)$.
Then
$$
m_G(1) \ge l-1.
$$
\end{theorem}

\begin{proof}
  Suppose that $G_1, \dots , G_l$ and $G$ are the graphs in the statement of the theorem.
  Let  $V_0$ be the vertices set of $\overline{K}_k$ and let $V_i$ be the vertices set of $G_i$ for $i=1,\ldots, l$.
  Our observation implies that the partition $\pi=(V_0,V_1,\ldots,V_l)$ is an equitable partition of $G$.
  Then the divisor matrix $L^\pi(G)$ is equal to
  \begin{equation*}
       \left(\begin{array}{ccccc}
            n  & -n_1 & -n_2 & \cdots & -n_l \\
            -1 & 1 & 0  & \cdots & 0 \\
            -1 & 0 & 1 & \cdots & 0 \\              
            \vdots & \vdots & \vdots & \ddots & \vdots \\
            -1 & 0 & 0 & \cdots & 1 \\
             \end{array}\right),
   \end{equation*}
  where $n_i=\frac{|V_i|}{k}$ for $i=1,\ldots, l$ and $n=\sum_{i=1}^l n_i$. 
  We can partition the matrix $xI-L^\pi(G)$ into four blocks as
  \begin{equation*}
       \left(\begin{array}{c:cccc}
                    x-n  & n_1 & n_2 & \cdots & n_l \\ \hdashline
                    1 & x - 1 & 0 & \cdots & 0 \\
                    1 & 0 & x -1 & \cdots & 0 \\
                    \vdots & \vdots & \vdots & \ddots & \vdots \\
                    1 & 0 & 0 & \cdots & x -1 \\
                \end{array}\right).
  \end{equation*}
  Then the characteristic polynomial of $L^\pi(G)$ is
  \begin{align*}
    \mu(L^\pi(G),x)
        &=(x-1)^l\bigg((x-n)-\frac{1}{x-1}n\bigg)\\
        &=(x-1)^{l-1}((x-n)(x-1)-n)\\ 
        &=(x-1)^{l-1}x(x-(n+1)).
  \end{align*}
  By Lemma \ref{lem:equit}, we obtain $m_G(1) \ge l-1$.
\end{proof}

\begin{theorem} \label{thm:2-conn}
  Let $G_1, \dots , G_l$ be connected $k$-symmetric graphs for some integers $l, k \geq 2$.
  Then the graph $G=\overline{K}_k \vee_k \left( G_1 \cup \cdots \cup G_l \right)$ is a 2-connected prime graph with respect to Cartesian product.
\end{theorem}

\begin{proof}
  First we prove that the graph $G=\overline{K}_k \vee_k \left( G_1 \cup \cdots \cup G_l \right)$ is 2-connected.
  Suppose that there is a cut-vertex $v$ of $G$ in $\overline{K}_k$.
  Since $k \geq 2$,
  there is another vertex $w$ in $\overline{K}_k$.
  Since $l \geq 2$,
  there are two independent paths $P_1$ and $P_2$ from $v$ to $w$ passing through $G_1$ and $G_2$, respectively.
  This implies that $v$ lies on the cycle $P_1 \cup P_2$,
  and hence $v$ is not a cut-vertex.
  Next we suppose that a cut-vertex of $G$ is not lying on $\overline{K}_k$.
  Without loss of generality, assume that a cut-vertex $v$ is in $G_1$.
  Since $l, k \geq 2$ and $G$ is connected,
  there is a cycle containing $v$ in $G$.
  It follows that $v$ is not a cut-vertex.
  Therefore, $G$ is 2-connected.
  
  Now, we show that the graph $G$ is prime with respect to the Cartesian product.
  It is well-known that if two edges of a nontrivial Cartesian product are incident,
  then they are included in a subgraph $C_4$ of the Cartesian product.
  For any vertex $u$ of $G$ in $\overline{K}_k$,
  two incident edges $e$ and $f$ of $u$
  such that the endpoints of $e$ and $f$ are contained in different graphs $G_i$ and $G_j$ for some integers $i$ and $j$.
  But there is no $C_4$ including $e$ and $f$.
  Thus we deduce that $G$ is prime.
\end{proof}

Theorems~\ref{thm:orbit} and \ref{thm:2-conn} imply Theorem~\ref{thm:main}.
Note that, if one of the graphs $G_1, \ldots, G_l$ is disconnected, then there is a counterexample.
For example, if $G_1$ is $\overline{K}_k$,
then $G=\overline{K}_k \vee_k \left( G_1 \cup \cdots \cup G_l \right)$ has a cut vertex.

\section{Laplacian integral graphs} \label{sec:laplacian}

In this section, we discuss $k$-symmetric graphs with integral Laplacian spectrum.
Let $n$ and $m$ be positive integers. 
Since the $n$-complete graph $K_n$ is $n$-symmetric, the disjoint union of $m$ copies of $K_n$, denoted by $mK_n$, is also $n$-symmetric.
We consider the $n$-symmetric join of $\overline{K}_n$ and $mK_n$.
Denote the graph $\overline{K}_n \vee_n mK_n$ by $\mathcal{C}(n,m)$.
Now we observe that a graph $C(n,m)$ is not a cograph for $n\geq 2$. 
Let $v$ and $w$ be vertices in $\overline{K}_n$.
Then there are two vertices in $K_n$ which are adjacent to $v$ and $w$, respectively.
Thus the graph $C(n,m)$ contains the path $P_4$ as an induced subgraph. 
We will show that a graph $\mathcal{C}(n,m)$ is Laplacian integral for some positive integers $n$ and $m$.
In the following theorem, we give the characteristic polynomial of $L(\mathcal{C}(n,m))$.
\begin{theorem}\label{thm:char}
  Let $n$ and $m$ be positive integers. Then the characteristic polynomial of $L(\mathcal{C}(n,m))$ is 
  \begin{equation*}
      x(x-1)^{m-1}(x-(n+1))^{(m-1)(n-1)}(x-(m+1))(x^2-(m+n+1)x+mn)^{n-1}.
  \end{equation*}
\end{theorem}
\begin{proof}
  The Laplacian matrix of $\mathcal{C}(n,m)$ is
  \begin{equation*}
      L(\mathcal{C}(n,m))=
      \begin{pmatrix}
          mI_n+L(\overline{K}_n) &  -1_m^T \otimes I_n\\
          -1_m \otimes I_n & I_m \otimes (I_n+L(K_n))
      \end{pmatrix}.
  \end{equation*}
  We consider $xI_{n(m+1)}-L(\mathcal{C}(n,m))$ as a matrix over the field of rational functions $\mathbb{C}(x)$.
  Then the characteristic polynomial  of $L(\mathcal{C}(n,m))$ is
  \begin{align*}
      \mu(L(\mathcal{C}(n,m)),x)=&\det(xI_{n(m+1)}-L(\mathcal{C}(n,m)))\\
      =&\det
      \begin{pmatrix}
          (x-m)I_n &  1_m^T \otimes I_n\\
          1_m \otimes I_n & I_m \otimes ((x-1)I_n-L(K_n))
      \end{pmatrix}\\
      =&\det\big(I_m \otimes ((x-1)I_n-L(K_n))\big)\det\big((x-m)I_n\\
      &~~-(1_m^T \otimes I_n)(I_m \otimes ((x-1)I_n-L(K_n)))^{-1}(1_m \otimes I_n)\big)\\
      =&\det((x-1)I_n-L(K_n))^m\det\big((x-m)I_n\\
      &~~-m((x-1)I_n-L(K_n))^{-1}\big).
  \end{align*}
  Since $\det(xI_n-L(K_n))=\mu(L(K_n),x)$, we obtain 
  \begin{align*}
      \det((x-1)I_n-L(K_n))^m&=\mu(L(K_n),x-1)^m\\
      &=(x-1)^m(x-(n+1))^{m(n-1)}.
  \end{align*}
   Now, we compute the determinant of $(x-m)I_n-m((x-1)I_n-L(K_n))^{-1}.$
   By Corollary \ref{coro:detinv} (b), we have
  \begin{align*}
      \big((x-1)I_n-L(K_n)\big)^{-1}&=\big((x-(n+1))I_n+J_n\big)^{-1}\\
      &=\frac{1}{(x-1)(x-(n+1))}\big((x-1)I_n-J_n\big).
  \end{align*}
  This implies that
  \begin{align*}
      &~\det\big((x-m)I_n-m((x-1)I_n-L(K_n))^{-1}\big)\\
      =&~ \frac{\det\big((x-m)(x-1)(x-(n+1))I_n-m((x-1)I_n-J_n)\big)}{(x-1)^n(x-(n+1))^n}\\
      =&~\frac{\det\big((x^3-(m+n+2)x^2+(mn+m+n+1)x-mn)I_n+mJ_n\big)}{(x-1)^n(x-(n+1))^n}
  \end{align*}
  By Corollary \ref{coro:detinv} (a), we have
  \begin{align*}
        &\det\big((x^3-(m+n+2)x^2+(mn+m+n+1)x-mn)I_n+mJ_n\big)\\
      =&~x(x-(m+1))(x-(n+1))(x-1)^{n-1}(x^2-(m+n+1)x+mn)^{n-1}.
  \end{align*}
  Hence the determinant of $(x-m)I_n-m((x-1)I_n-L(K_n))^{-1}$ is
  \begin{equation*}
      \frac{x(x-(m+1))(x^2-(m+n+1)x+mn)^{n-1}}{(x-(n+1))^{n-1}(x-1)}.
  \end{equation*}
  Therefore the characteristic polynomial of $L(\mathcal{C}(n,m))$ is
   \begin{multline*}
      x(x-1)^{m-1}(x-(n+1))^{(m-1)(n-1)}(x-(m+1))(x^2-(m+n+1)x+mn)^{n-1}.
  \end{multline*}
\end{proof}

The following corollary induces Theorem~\ref{thm:main2}.
\begin{corollary}
  Let $n$, $m$, $k$, and $l$ be positive integers with $l\neq 1$. Then
  \begin{itemize}
      \item[(a)] If $\mathcal{C}(n,m)$ is Laplacian integral, then $\mathcal{C}(m,n)$ is also Laplacian integral.
      \item[(b)] A graph $\mathcal{C}(kl,(k+1)(l-1))$ is Laplacian integral.
      \item[(c)] A graph $\mathcal{C}(k^2+k,k^2+k)$ is regular Laplacian integral.
  \end{itemize}
\end{corollary}
\begin{proof}
 \begin{itemize}
     \item[(a)] It is obvious by Theorem \ref{thm:char}.
     \item[(b)] If the quadratic $x^2-(m+n+1)x+mn$ has two integer roots, then $\mathcal{C}(m,n)$ is Laplacian integral, by Theorem \ref{thm:char}.
    Let $k$, $l$, $r$ and $s$ be positive integers with $n=kl$ and $m=rs$.
    Suppose that $kr$ and $ls$ are roots of the quadratic. 
    Then, by Vieta's formulas, we have $kr+ls=rs+kl+1$, that is,
    \begin{equation}\label{eq:rskl}
       (s-k)r-(s-k)l+1=0. 
    \end{equation}
    If $s=k$ then it is a contradiction.
    If $s-k \neq 0$, then $r-l+\frac{1}{s-k}=0$. 
    Since $r$ and $l$ are integers, $s-k$ must be 1. 
    Plugging $s=k+1$ into the equation (\ref{eq:rskl}), we have $r=l-1$. 
    Since $m$ is a positive integer, $l$ is not equal to $1$.
    Thus $\mathcal{C}(kl,(k+1)(l-1))$ is Laplacian integral for any positive integers $k$ and $l\neq 1$.
    
    \item[(c)] If $m=n$, then $\mathcal{C}(n,m)$ is regular. By (b), a graph $\mathcal{C}(k^2+k,k^2+k)$ is regular Laplacian integral graph.
 \end{itemize}
\end{proof}

Now, we consider the $n$-complete graph $K_n$ as a $k$-symmetric graph for some divisor $k$ of $n$. 
Note that a base of $K_n$ as a $k$-symmetric graph is not unique, but the $k$-symmetric join of $\overline{K}_k$ and $mK_n$ is unique up to isomorphism.  
We denote by $\mathcal{C}(n,k,m)$ the graph $\overline{K}_k \vee_k mK_n$.
In the similar way to the proof of Theorem~\ref{thm:char}, we get the characteristic polynomial of $L(\mathcal{C}(n,k,m))$.

\begin{theorem}\label{thm:char-k}
  Let $n$ and $m$ be positive integers. Let $k$ be a divisor of $n$ and let $d=n/k$.
  Then the characteristic polynomial of $L(\mathcal{C}(n,k,m))$ is 
  \begin{equation*}
      x(x-1)^{m-1}(x-(n+1))^{m(n-1)-k+1}(x-(md+1))(x^2-(md+n+1)x+mdn)^{k-1}.
  \end{equation*}
\end{theorem}
\begin{proof}
  The Laplacian matrix of $\mathcal{C}(n,k,m)$ is 
  \begin{equation*}
      L(\mathcal{C}(n,k,m))=\begin{pmatrix}
       mdI_k+L(\overline{K}_k) & -1_m^T\otimes (I_k \otimes 1_d^T)\\
       -1_m\otimes (I_k \otimes 1_d) & I_m \otimes(I_n+L(K_n))
      \end{pmatrix}.
  \end{equation*}
  Then the characteristic polynomial of $L(\mathcal{C}(n,k,m))$ is
  \begin{equation*}
      \mu(L(\mathcal{C}(n,k,m)),x)=\det\big(xI_{n(m+1)}-L(\mathcal{C}(n,k,m))\big).
  \end{equation*}
  Consider $xI_{n(m+1)}-L(\mathcal{C}(n,k,m))$ as a matrix over the field of rational functions $\mathbb{C}(x)$. Then
  \begin{align*}
      &~~~\det\big(xI_{n(m+1)}-L(\mathcal{C}(n,k,m))\big)\\
      &=\det\begin{pmatrix}
       (x-md)I_k & 1_m^T\otimes (I_k \otimes 1_d^T)\\
       1_m\otimes (I_k \otimes 1_d) & I_m \otimes((x-1)I_n-L(K_n))
      \end{pmatrix}\\
      &=\det\big(I_m \otimes((x-1)I_n-L(K_n))\big) \det\big((x-md)I_k\\
      &~~-(1_m^T\otimes (I_k \otimes 1_d^T))(I_m \otimes((x-1)I_n-L(K_n)))^{-1}(1_m\otimes (I_k \otimes 1_d) \big)\\
      &=\det\big((x-1)I_n-L(K_n)\big)^m\det\big((x-md)I_k\\
      &~~-m\big(I_k \otimes 1_d^T)((x-1)I_n-L(K_n))^{-1}(I_k \otimes 1_d)\big).
  \end{align*}
  It is easily check that 
  \begin{equation*}
      \det\big((x-1)I_n-L(K_n)\big)^m=(x-1)^m(x-(n+1))^{m(n-1)}.
  \end{equation*}
  Now, we compute $\det\big((x-md)I_k-m\big(I_k \otimes 1_d^T)((x-1)I_n-L(K_n))^{-1}(I_k \otimes 1_d)\big).$
  By Corollary \ref{coro:detinv} (b), we have 
  \begin{equation*}
      \big((x-1)I_n-L(K_n)\big)^{-1}=\frac{1}{(x-1)(x-(n+1))}((x-1)I_n-J_n).
  \end{equation*}
  Note that the matrix $(x-1)I_n-J_n$ can be written in the Kronecker product form $(x-1)I_k\otimes I_d -J_k \otimes J_d$. It follows that 
  \begin{align*}
      &~(I_k \otimes 1_d^T)((x-1)I_n-L(K_n))^{-1}(I_k \otimes 1_d)\\
      =&~(x-1)^{-1}(x-(n+1))^{-1}(I_k \otimes 1_d^T)((x-1)I_k\otimes I_d -J_k \otimes J_d)(I_k \otimes 1_d)\\
      =&~(x-1)^{-1}(x-(n+1))^{-1}(d(x-1)I_k-d^2J_k).
  \end{align*}
  Then we have
  \begin{align*}
      &\det\big((x-md)I_k-m \big(I_k \otimes 1_d^T)((x-1)I_n+L(K_n))^{-1}(I_k \otimes 1_d)\big)\big)\\
      =&\frac{\det\big((x-md)(x-1)(x-(n+1))I_k- md(x-1)I_k+md^2J_k \big)}{(x-1)^k(x-(n+1))^k}.\\
  \end{align*}
  By Corollary \ref{coro:detinv} (a), we obtain
  \begin{align*}
      &\det\big((x-md)(x-1)(x-(n+1))I_k- md(x-1)I_k+md^2J_k \big)\\
      =&\det\big( (x^3-(md+n+2)x^2+(md+1)(n+1)x-mdn)I_k+md^2J_k \big)\\
      =&~x(x-(md+1))(x-(n+1))(x-1)^{k-1}(x^2-(md+n+1)x+mdn)^{k-1}.
  \end{align*}
  Hence the determinant of $(x-md)I_k-m\big(I_k \otimes 1_d^T)((x-1)I_n-L(K_n))^{-1}(I_k \otimes 1_d)$ is
  \begin{equation*}
      \frac{x(x-(md+1))(x^2-(md+n+1)x+mdn)^{k-1}}{(x-1)(x-(n+1))^{k-1}}.
  \end{equation*}
  Thus the characteristic polynomial of $L(\mathcal{C}(n,k,m))$ is
  \begin{align*}
      x(x-1)^{m-1}(x-(n+1))^{m(n-1)-k+1}(x-(md+1))(x^2-(md+n+1)x+mdn)^{k-1}.
  \end{align*}
\end{proof}

The next two corollaries tell us about the relation between Laplacian integral graphs $\mathcal{C}(n,n,m)$ and $\mathcal{C}(n,k,m')$ for some positive integers $n$, $m$, $m'$ and $k$ with $k\,|\,n$.

\begin{corollary}\label{coro:n-k}
  Suppose that $\mathcal{C}(n,n,m)$ is Laplacian integral for some positive integers $n$ and $m$.
  Let $d$ be a divisor of $n$.
  If $m$ is divisible by $d$, then $\mathcal{C}(n,\frac{n}{d},\frac{m}{d})$ is Laplacian integral.
\end{corollary}
\begin{proof}
  Suppose that $\mathcal{C}(n,n,m)$ is Laplacaian integral for some positive integers $n$ and $m$.
  Then the polynomial $x^2-(m+n+1)x+mn$ in the characteristic polynomial of $\mathcal{C}(n,n,m)$ can be factored over the integers.
  Let $d$ be a divisor of $n$.
  By Theorem \ref{thm:char-k}, it is enough to show that the quadratic in the characteristic polynomial of $\mathcal{C}(n,\frac{n}{d},\frac{m}{d})$ has integral roots.
  Since the quadratic is
  \begin{equation*}
      x^2-\bigg(\frac{m}{d}dn+n+1\bigg)x+\frac{m}{d}d n =x^2-(m+n+1)x+mn,
  \end{equation*}
  the graph $\mathcal{C}(n,\frac{n}{d},\frac{m}{d})$ is Laplacian integral.
\end{proof}

\begin{corollary}\label{coro:k-n}
  Suppose that $\mathcal{C}(n,k,m)$ is Laplacian integral for some positive integers $n$, $m$ and $k$ with $k\,|\,n$. Let $d=n/k$. Then $\mathcal{C}(n,n,md)$ is Laplacian integral.
\end{corollary}
\begin{proof}
  The proof is similar that of Theorem \ref{coro:n-k}. Since the quadratic in the characteristic polynomial of $\mathcal{C}(n,n,md)$ is
  \begin{equation*}
      x^2-(md+n+1)x+mdn,
  \end{equation*}
  it is easy to see that $\mathcal{C}(n,n,md)$ is Laplacian integral.
\end{proof}

\section*{Declaration of Competing Interest}
There is no competing interest.

\end{document}